\documentclass[12 pt,twoside]{amsart}

\usepackage{amsopn}
\usepackage{amssymb}
\usepackage{amscd}

\newtheorem{theorem}{Theorem}[section]

\newtheorem{proposition}[theorem]{Proposition}

\theoremstyle{definition}
\newtheorem{definition}[theorem]{Definition}

\theoremstyle{remark}

\numberwithin{equation}{section}

\begin{document}

\title[Limit and extended limit sets of matrices]{Limit and extended limit sets of matrices in Jordan normal form}

\author[George Costakis]{George Costakis}
\address{Department of Mathematics, University of Crete, Knossos Avenue, GR-714 09 Heraklion, Crete, Greece}
\email{costakis@math.uoc.gr}
\thanks{}

\author[Antonios Manoussos]{Antonios Manoussos}
\address{Fakult\"{a}t f\"{u}r Mathematik, SFB 701, Universit\"{a}t Bielefeld, Postfach 100131, D-33501 Bielefeld, Germany}
\email{amanouss@math.uni-bielefeld.de}
\thanks{During this research the second author was fully supported by SFB 701 ``Spektrale Strukturen und
Topologische Methoden in der Mathematik" at the University of Bielefeld, Germany. He would also like to express
his gratitude to Professor H. Abels for his support.}


\date{}


\begin{abstract}
In this note we describe the limit and the extended limit sets of every vector for a single matrix in Jordan normal form.
\end{abstract}

\maketitle

\section{Preliminaries and basic notions}
Limit and extended limit sets are in the center of interest in the study of dynamics of linear operators. To find them, even in relatively easy cases of operators, it
is a difficult task. In this note we describe the limit and the extended limit sets for the simplest case which is the case of a
single matrix in Jordan normal form. We use a method similar  to the one used by N. H. Kuiper and J. W. Robbin in \cite{KuRo}. In this work Kuiper and Robbin dealt
with the problem of the topological classification of linear endomorphisms and the main tool they used was the extended mixing limit sets of the exponential of the
nilpotent part of a Jordan block. In the following we introduce the basic notions we use in the present work.
\medskip

\noindent Let $X$ be a complex Banach space and let $T:X\to X$ be a bounded linear operator.

\begin{definition}
For every $x\in X$ the sets
\[
\begin{split}
L(x)=\{&y\in X:\,\mbox{ there exists a strictly increasing sequence}\\
       &\mbox{of positive integers}\,\,\{k_{n}\}\mbox{ such that }\,T^{k_{n} }x\rightarrow y\}
\end{split}
\]
\[
\begin{split}
J(x)=\{ &y\in X:\,\mbox{there exist a strictly increasing sequence of positive}\\
&\mbox{integers}\,\{k_{n}\}\,\mbox{and a sequence }\,\{x_{n}\}\subset X\,\mbox{such that}\,\, x_{n}\rightarrow x\,
\mbox{and}\\
&T^{k_{n}}x_{n}\rightarrow y\}
\end{split}
\]
and
\[
\begin{split}
J^{mix}(x)=\{ &y\in X:\,\mbox{ there exists a sequence}\,\{x_{n}\}\subset X\,\mbox{such that}\\
&x_{n}\rightarrow x\,\,\mbox{and}\,\,T^{n}x_{n}\rightarrow y\}
\end{split}
\]
denote the limit set, the extended (prolongational) and the extended mixing limit set of $x$ under $T$ respectively.
\end{definition}

The notions of the limit and extended limit sets are well known in the theory of topological dynamics, see
\cite{BaSz}. Roughly speaking, we can say that the limit set of a vector $x$ describes the limiting behavior of
its orbit and the corresponding extended limit set $J(x)$ describes the asymptotic behavior of all vectors
nearby $x$. Let us see how these notions are connected, at a very first, naive level. To explain, we recall the
following definition. A bounded linear operator $T:X\to X$ acting on a complex separable Banach space $X$ is
called \textit{hypercyclic} if there exists a vector $x\in X$ so that the orbit of $x$ under $T$, i.e. the set
$\{ x, Tx, T^2x, \ldots \}$, is dense in $X$. The last is equivalent with the following property: for every pair
$(U,V)$ of non-empty open subsets of $X$ there exists a positive integer $n$ such that $T^nU\cap V\neq
\emptyset$. In the language of topological dynamics an operator having the previous property is called
\textit{topologically transitive}. If in addition for every pair $(U,V)$ of non-empty open subsets of $X$ there
exists a positive integer $N$ such that $T^nU\cap V\neq \emptyset$ for every $n\geq N$, then $T$ is called
\textit{topologically mixing}. Using Proposition \ref{equivdef} it is not difficult to show that $T$ is
topologically transitive  if and only if $J(x)=X$ for every $x\in X$ and that $T$ is topologically mixing if and
only if $J^{mix}(x)=X$ for every $x\in X$ see \cite{cosma1}. For several examples of hypercyclic operators and
an in depth study of several aspects of the notion of hypercyclicity we refer to the recent books \cite{BaMa},
\cite{GEPe}. In \cite{cosma1}, we localized the notion of hypercyclicity through the use of $J$-sets. To recall
briefly, we say that a bounded linear operator $T$ acting on a Banach space (not necessarily separable) $X$ is
\textit{locally topologically transitive} or \textit{$J$-class} (\textit{locally topologically mixing}) if there
exists a non-zero vector $x\in X$ so that $J(x)=X$ ($J^{mix}(x)=X$ respectively). Among other things we showed
that there are locally hypercyclic, non-hypercyclic operators and that finite dimensional Banach spaces do not
admit locally hypercyclic operators. The next proposition, which also appears in \cite{CosMa}, gives a
description of $J$-sets through the use of open sets. To keep the paper shelf contained we present the proof.

\begin{proposition} \label{equivdef}
An equivalent definition for the sets $J(x)$, $J^{mix}(x)$ is the following.
\[
\begin{split}
J(x)=\{ &y\in X:\,\mbox{for every pair of neighborhoods}\,\, U,V\,\mbox{of}\,\, x,y\\
        &\mbox{respectively, there exists a positive integer}\,\,n,\\
        &\mbox{such that}\,\, T^{n}U\cap V\neq\emptyset \}.
\end{split}
\]
and
\[
\begin{split}
J^{mix}(x)=\{ &y\in X:\,\mbox{for every pair of neighborhoods}\,\, U,V\,\mbox{of}\,\, x,y\\
                  &\mbox{respectively, there exists a positive integer}\,\, N,\\
                  &\mbox{such that}\,\, T^{n}U\cap V\neq\emptyset\,\,\mbox{for every}\,\, n\geq N\}.
\end{split}
\]
\end{proposition}
\begin{proof}
We give the proof for the $J^{mix}$-sets, the proof for the $J$-sets is similar.

Let us prove that
\[
\begin{split}
J^{mix}(x) \supset \{ &y\in X:\,\mbox{for every pair of neighborhoods}\,\, U,V\,\mbox{of}\,\, x,y\\
                  &\mbox{respectively, there exists a positive integer}\,\, N,\\
                  &\mbox{such that}\,\, T^{n}U\cap V\neq\emptyset\,\,\mbox{for every}\,\, n\geq N\}.
\end{split}
\]
since the converse inclusion is obvious. Fix a vector
\[
\begin{split}
y\in \{ &y\in X:\,\mbox{for every pair of neighborhoods}\,\, U,V\,\mbox{of}\,\, x,y\\
                  &\mbox{respectively, there exists a positive integer}\,\, N,\\
                  &\mbox{such that}\,\, T^{n}U\cap V\neq\emptyset\,\,\mbox{for every}\,\, n\geq N\}.
\end{split}
\]
Consider the open balls $B(x,1/n), B(y,1/n)$ centered at $x,y\in X$ and radius $1/n$ for $n=1,2,\ldots$. Then
there exists a strictly increasing sequence $\{ k_n \}$ of positive integers such that $ T^m B(x,1/n)\bigcap
B(y,1/n)\neq\emptyset$ for every $m\geq k_n$, $n=1,2,\ldots$. Therefore there exist $$x_{k_1}, x_{k_1+1}, \ldots
,x_{k_2-1}\in B(x,1)$$ such that $ \| T^mx_m-y\| <1 $ for every $ m=k_1, k_1+1, \ldots ,k_2-1$. In a similar
fashion we may find $$x_{k_2}, x_{k_2+1}, \ldots ,x_{k_3-1}\in B(x,1/2)$$ such that $ \| T^mx_m-y\| <1/2$ for
every $m=k_2, k_2+1, \ldots ,k_3-1$. Proceeding inductively we find a sequence $\{ x_n \}$, $n\geq k_1$ such
that $x_n\to x$ and $T^nx_n\to y$. This completes the proof.
\end{proof}

Limit sets, extended limit sets and extended mixing limit sets are closed and invariant \cite{CosMa}. Next
proposition will be used later to simplify proofs.

\begin{proposition} \label{scalarA}
Let $T:X \rightarrow X$ be a bounded linear operator. Then $J_{\lambda T}(0)=J_{T}(0)$ for every $|\lambda |=1$.
\end{proposition}
\begin{proof}
Let $y\in J_{\lambda T}(0)$. Then there exist a strictly increasing sequence of positive integers $\{ k_{n}\}$ and a sequence $\{ x_{n}\}\subset X$ such
that $x_{n}\to 0$ and  $\lambda^{k_n}T^{k_{n}}x_{n}\to y$. Since $|\lambda |=1$ then $\lambda^{k_n}x_n\to 0$ and since $T^{k_{n}}(\lambda^{k_n}x_{n})\to y$ it follows
that $y\in J_{T}(0)$. Take now a vector $y\in J_{T}(0)$. Then there exist a strictly increasing sequence of positive integers $\{ k_{n}\}$ and a sequence $\{
x_{n}\}\subset X$ such that $x_{n}\to 0$ and  $T^{k_{n}}x_{n}\to y$. Since $|\lambda |=1$, without loss of generality we may assume that
$\lambda^{k_n}\to\mu$ for some $|\mu |=1$. Hence $\lambda^{k_n}T^{k_{n}}\frac{x_{n}}{\mu}\to y$ and since $\frac{x_{n}}{\mu}\to 0$ then $y\in J_{\lambda T}(0)$.
\end{proof}

\section{Limit and extended limit sets of a matrix in Jordan normal form}
We mainly focus in the case $A$ is a $l\times l$ Jordan block over $\mathbb{C}$. This means that the main diagonal consists of $\lambda$'s, for some $\lambda \in
\mathbb{C}$, the diagonal above the main diagonal consists of $1$'s and all the other entries of the matrix are filled with zeros. We shall then describe the limit
and extended limit sets of every $x\in \mathbb{C}^l$ under $A$. The general case of a matrix in a Jordan canonical form follows easily from the latter case, since we
can ``glue" the limit and extended limit sets of separate Jordan blocks. Finally, since every complex matrix $B$ is similar to a matrix in Jordan canonical form, we
are able to determine the limit and extended limit sets of every $x\in \mathbb{C}^l$ under $B$. For the rest of this section $A$ will be a $l\times l$ Jordan block
over $\mathbb{C}$.

\begin{proposition}
\mbox{}

\begin{itemize}
\item[(i)] If $A$ is a Jordan block with an eigenvalue $|\lambda | =1$ and $x=(x_1,\ldots ,x_l)\in \mathbb{C}^l$
then $L(x)\neq\emptyset$ if and only if $x_{2}=\ldots =x_{l}=0$. In this case $L(x)=\{ Dx_{1}\}\times 0$, where
$D$ is the closure of the set $\{ \lambda^{n}\}$.

\item[(ii)] If $A$ is a Jordan block with an eigenvalue $|\lambda | >1$ then $L(x)=\emptyset$ for every non-zero vector $x\in \mathbb{C}^{l}$.

\item[(iii)] If $A$ is a Jordan block with an eigenvalue $|\lambda | <1$ then $L(x)=\{ 0\}$ for every $x\in \mathbb{C}^{l}$.
\end{itemize}
\end{proposition}
\begin{proof}
(i) We give the proof for the case $l=3$. The general case follows easily. Let $x=(x_{1},x_{2},x_{3})$ and $y=(y_{1},y_{2},y_{3})\in L(x)$. Then there exists a
strictly increasing sequence of positive integers $\{ k_n\}$ such that $A^{k_{n}}x\rightarrow y$. Setting $y_{n}=(y_{n1},y_{n2},y_{n3})=A^{k_{n}}x$ we get
\[
\begin{array}{l}
y_{n1}=\lambda^{k_{n}}x_{1} + k_{n}\lambda^{k_{n}-1}x_{2} + \frac{k_{n}(k_{n}-1)}{2}\lambda^{k_{n}-2}x_{3}\\
y_{n2}=\lambda^{k_{n}}x_{2} + k_{n}\lambda^{k_{n}-1}x_{3}\\
y_{n3}=\lambda^{k_{n}}x_{3}.
\end{array}
\]
Using the second equation we conclude that $x_{3}=y_{3}=0$ since the sequences $\{ y_{n2}\}$ and $\{ \lambda^{k_{n}}x_{2}\}$ are bounded. So we have the following
system of linear equations.
\[
\begin{array}{l}
y_{n1}=\lambda^{k_{n}}x_{1} + k_{n}\lambda^{k_{n}-1}x_{2}\\
y_{n2}=\lambda^{k_{n}}x_{2}.
\end{array}
\]
Using the same argument for $x_{2}$ we have $x_{2}=y_{2}=0$. Hence, the only remaining equation is $y_{n1}=\lambda^{k_{n}}x_{1}$ and the proof of the proposition is
completed.

The proof of items (i) and (ii) follows easily and is omitted.
\end{proof}

In the following we describe the extended limit sets of the zero-vector.

\begin{theorem}
Let $A$ be a Jordan block with an eigenvalue $|\lambda | =1$.
\begin{itemize}
\item[(i)]   If $l=1$ then $J(0)=\{ 0\}$. If $l>1$ and $l$ is of the form $l=2r$ or $l=2r-1$ then
$J(0)=\mathbb{C}^{r}\times 0$.

\item[(ii)]  A point $y\in J(0)$ if and only if there exists a sequence $\{x_{n}\}\subset {\mathbb{C}}^l $ such that $A^{n}x_{n}\rightarrow y$, hence
$J(0)=J^{mix}(0)$.

\item[(iii)] For every linear map
$B:\mathbb{C}^{l} \rightarrow \mathbb{C}^{l}$ with eigenvalues of modulus $1$ the set $J(0)$ is a proper linear subspace of $\mathbb{C}^{l}$.
\end{itemize}
\end{theorem}
\begin{proof}
We give the proof for the cases $l=3$ and $l=4$. For the general case we may use the same technics as in \cite{KuRo}. Since, by Proposition \ref{scalarA}, $J_{\lambda
A}(0)=J_{A}(0)$ for every $\lambda$ of modulus $1$ we may assume that $\lambda=1$.

\medskip

\noindent \textit{Case} $l=3$: Let $y=((y_{1},y_{2},y_{3})\in J(0)$. Then there exist a sequence $\{x_n \}$ in
$\mathbb{C}^3$ and a strictly increasing sequence $\{ k_n\}$ of positive integers such that
$x_{n}=(x_{n1},x_{n2},x_{n3})\rightarrow (0,0,0)$ and $A^{k_{n}}x_{n}\rightarrow y$. Letting
$y_{n}=(y_{n1},y_{n2},y_{n3})=A^{k_{n}}x_{n}$, then we have
\[
\begin{array}{l}
y_{n1}=x_{n1} + k_{n}x_{n2} + \frac{k_{n}(k_{n}-1)}{2}x_{n3}\\
y_{n2}=x_{n2} + k_{n}x_{n3}\\
y_{n3}=x_{n3}.
\end{array}
\]
Since $x_{n3}\rightarrow 0$ it follows that $y_{n3}\rightarrow 0$, hence $y_{3}=0$. Using the second equation we
have $x_{n3}=\frac{y_{n2}-x_{n2}}{k_{n}}$. Therefore dividing the first equation by $k_{n}$ and substitute
$x_{n3}$ we get
\[
\frac{y_{n1}}{k_{n}}=\frac{x_{n1}}{k_{n}}+x_{n2}+\frac{k_{n}-1}{2k_{n}}(y_{n2}-x_{n2}).
\]
Since $\frac{y_{n1}}{k_{n}}$, $\frac{x_{n1}}{k_{n}}$ and $x_{n2}$ have limit $0$ and
$\frac{k_{n}-1}{2k_{n}}\rightarrow \frac{1}{2}$ then $y_{n2}-x_{n2}\rightarrow 0$. Since $x_{n2}\rightarrow 0$
it follows that $y_{n2}\rightarrow 0$, therefore $y_{2}=0$. Till now we have proved that
$J(0)\subset\mathbb{C}\times 0$. Next we show the inverse inclusion. Let $y=(y_{1},0,0)$. We put $x_{n1}=0$,
$x_{n2}=0$, $y_{n1}=y_{1}$ and then we solve the system of the linear equations. So, we have
\[
\begin{array}{lll}
x_{n1}=0, &x_{n2}=0, &x_{n3}=\frac{2y_{1}}{n(n-1)}\\
y_{n1}=y_{1}, &y_{n2}=\frac{2y_{1}}{n-1}, &y_{n3}=\frac{2y_{1}}{n(n-1)}.
\end{array}
\]
Now it is easy to check that $x_{n}\rightarrow 0$ and $y_{n}\rightarrow y$. Note that we have also proved (ii).

\medskip

\noindent \textit{Case} $l=4$: Let $y=((y_{1},y_{2},y_{3}, y_{4})\in J(0)$. Then there exist a sequence $\{x_n
\}$ in $\mathbb{C}^4$ and a strictly increasing sequence $\{ k_n\}$ of positive integers such that
$x_{n}=(x_{n1},x_{n2},x_{n3},x_{n4})\rightarrow (0,0,0)$ and $A^{k_{n}}x_{n}\rightarrow y$. Setting
$y_{n}=(y_{n1},y_{n2},y_{n3},y_{n4})=A^{k_{n}}x_{n}$ we get
\[
\begin{array}{l}
y_{n1}=x_{n1} + k_{n}x_{n2} + \frac{k_{n}(k_{n}-1)}{2}x_{n3} + \frac{k_{n}(k_{n}-1)(k_{n}-2)}{6}x_{n4}\\
y_{n2}=x_{n2} + k_{n}x_{n3} + \frac{k_{n}(k_{n}-1)}{2}x_{n4}\\
y_{n3}=x_{n3} + k_{n}x_{n4}\\
y_{n4}=x_{n4}.
\end{array}
\]
Observe that the last three equations are exactly the same as in the case where $l=3$ hence $y_{3}=y_{4}=0$.
Therefore $J(0)\subset\mathbb{C}^{2}\times 0$. Next we show the inverse inclusion. Let $y=(y_{1},y_{2},0,0)$. We
put $x_{n1}=0$, $x_{n2}=0$, $y_{n1}=y_{1}$, $y_{n2}=y_{2}$ and then we solve the system of the linear equations.
So, we have
\[
\begin{array}{lll}
x_{n1}=0, &x_{n2}=0\\
x_{n3}=\frac{2(3y_{1}-(n-2)y_{2})}{n(n+1)}, &x_{n4}=\frac{2(3(n-1)y_{2}-6y_{1})}{n(n-1)(n+1)}\\
y_{n1}=y_{1}, &y_{n2}=y_{2}\\
y_{n3}=\frac{2(3y_{1}-(n-2)y_{2})}{n(n+1)} + k_{n}\frac{2(3(n-1)y_{2}-6y_{1})}{n(n-1)(n+1)},
&y_{n4}=\frac{2(3(n-1)y_{2}-6y_{1})}{n(n-1)(n+1)}.
\end{array}
\]
Now it is easy to check that $x_{n}\rightarrow 0$ and $y_{n}\rightarrow y$. Again we have proved simultaneously
(ii). Item (iii) is obtained by (i) and (ii), since we can glue the $J(0)$-sets of the Jordan blocks of $B$.
\end{proof}

\begin{theorem} \label{jzerogreater}
Let $A$ be a Jordan block with an eigenvalue $|\lambda | >1$. Then the following hold.
\begin{itemize}
\item[(i)] $J(0)=\mathbb{C}^{l}$.

\item[(ii)] For every point $y\in\mathbb{C}^{l}$ there exists a sequence $\{ x_{n}\}\subset \mathbb{C}^{l}$ such that $A^{n}x_{n}\rightarrow y$, hence
$J(0)=J^{mix}(0)$.

\item[(iii)] For every linear map $B:\mathbb{C}^{l} \rightarrow \mathbb{C}^{l}$ with eigenvalues of modulus greater than $1$ it holds that
$J(0)=\mathbb{C}^{l}$.
\end{itemize}
\end{theorem}
\begin{proof}
We prove the theorem for the case $l=3$. Let $y=(y_{1},y_{2},y_{3})$ and
$y_{n}=(y_{n1},y_{n2},y_{n3})=A^{n}x_{n}$. We put $y_{n1}=y_{1}$, $y_{n2}=y_{2}$, $y_{n3}=y_{3}$ and we solve
again the corresponding system of the linear equations:
\[
\begin{array}{l}
y_{n1}=\lambda^{n} x_{n1} + n\lambda^{n-1}x_{n2} + \frac{n(n-1)}{2}\lambda^{n-2}x_{n3}\\
y_{n2}=\lambda^{n} x_{n2} + n\lambda^{n-1}x_{n3}\\
y_{n3}=\lambda^{n} x_{n3}.
\end{array}
\]
Hence, we have
\[
\begin{array}{l}
x_{n1}=\frac{2\lambda^{2}y_{1}-n2\lambda y_{2} + n(n+1)y_{3}}{2\lambda^{n+2}}\\
x_{n2}=\frac{\lambda y_{2}-ny_{3}}{\lambda^{n+1}}\\
x_{n3}=\frac{y_{n3}}{\lambda^{n}}.
\end{array}
\]
Now it is trivial to check that $x_{n}\rightarrow 0$ and $y_{n}\rightarrow y$. Note that item (iii) follows by
(i) and (ii).
\end{proof}

The proof of the next proposition is easy and it is left to the reader.
\begin{proposition}
Let $A$ be a Jordan block with an eigenvalue $|\lambda | <1$. Then the following hold.
\begin{itemize}
\item[(i)] $J(0)=\{ 0\}$.

\item[(ii)] For every linear map $B:\mathbb{C}^{l} \rightarrow \mathbb{C}^{l}$ with eigenvalues of modulus less than $1$ it holds that $J(0)=\{ 0\}$.
\end{itemize}
\end{proposition}

Below we treat the general case.

\begin{theorem}
Let $A$ be a Jordan block with an eigenvalue $|\lambda | =1$ and $x=(x_1,\ldots ,x_l)\in \mathbb{C}^l$.
\begin{itemize}
\item[(i)]   If $l=2r$ then $J(x)\neq\emptyset$ if and only if $x_{r+1}=\ldots =x_{l}=0$. In this case
$J(x)=\mathbb{C}^{r}\times 0$.

\item[(ii)]  If $l=2r-1$ then $J(x)\neq\emptyset$ if and only if $x_{r+1}=\ldots =x_{l}=0$. In this case
$J(x)=\mathbb{C}^{r-1}\times \{ Dx_{r}\}\times 0$, where $D$ is the closure of the set $\{ \lambda^{n}\}$. In
case where $l=1$, $J(x)=\{ Dx\}$.
\end{itemize}
\end{theorem}
\begin{proof}
We only consider the case $l=3$. Let $y=(y_{1},y_{2},y_{3})\in J(x)$. Then there exist a sequence $\{ z_n\}$ in
$\mathbb{C}^3$ and a strictly increasing sequence $\{ k_n\}$ of positive integers such that
$z_{n}=(x_{n1},x_{n2},x_{n3})\rightarrow (x_{1},x_{2},x_{3})$ and $A^{k_{n}}z_{n}\rightarrow y$. Setting
$y_{n}=(y_{n1},y_{n2},y_{n3})=A^{k_{n}}z_{n}$ we get
\[
\begin{array}{l}
y_{n1}=\lambda^{k_{n}}x_{n1} + k_{n}\lambda^{k_{n}-1}x_{n2} + \frac{k_{n}(k_{n}-1)}{2}\lambda^{k_{n}-2}x_{n3}\\
y_{n2}=\lambda^{k_{n}}x_{n2} + k_{n}\lambda^{k_{n}-1}x_{n3}\\
y_{n3}=\lambda^{k_{n}}x_{n3}.
\end{array}
\]
Using the second equation we conclude that $x_{n3}\rightarrow 0$ since the sequences $\{ y_{n2}\}$ and $\{
\lambda^{k_{n}}x_{n2}\}$ are bounded. The last implies that $x_{3}=0$. Since
\[
\frac{y_{n1}}{k_{n}}=\frac{\lambda^{k_{n}}x_{n1}}{k_{n}} + \lambda^{k_{n}-1}x_{n2} +
\frac{(k_{n}-1)}{2}\lambda^{k_{n}-2}x_{n3}
\]
and $\frac{y_{n1}}{k_{n}}\rightarrow 0$, $\frac{\lambda^{k_{n}}x_{n1}}{k_{n}}\rightarrow 0$ we obtain the
following
\[
\lambda^{k_{n}-1}x_{n2} + \frac{(k_{n}-1)}{2}\lambda^{k_{n}-2}x_{n3}\rightarrow 0.
\]
Solving the second equation with respect to $x_{n3}$ and substitute above, we arrive at
\[
2\lambda^{k_{n}-1}x_{n2} + (k_{n}-1)\frac{y_{n2}-\lambda^{k_{n}-2}x_{n2}}{\lambda k_{n}}\rightarrow 0,
\]
or equivalently
\[
2\lambda^{k_{n}-1}x_{n2}+ \frac{k_{n}-1}{\lambda k_{n}}y_{n2} -
\frac{k_{n}-1}{k_{n}}\lambda^{k_{n}-1}x_{n2}\rightarrow 0.
\]
Since $\frac{k_{n}-1}{\lambda k_{n}}y_{n2}\rightarrow \frac{1}{\lambda}y_{2}$, without loss of generality we may
assume that $\lambda^{k_{n}-1}\rightarrow \mu$ for some $\mu$ of modulus $1$. Then every $y=(y_1, y_2, y_3)\in
J(x)$ satisfies the following
\[
y_{2} + \lambda\mu x_{2}=0.
\]
That is $y_{2}\in \{ Dx_{2}\}$. Hence, $J(x)\subset\mathbb{C}\times \{ Dx_{2}\}\times 0$. Next we show the
inverse inclusion. Let $y=(y_{1},y_{2},0)$ where $y_{2}\in \{ Dx_{2}\}$. Hence there exists a sequence of the
form $\{ \lambda^{k_{n}}\}$ such that $\lambda^{k_{n}-1}\rightarrow -\mu$ for some $\mu$ of modulus $1$ and
$\lambda^{k_{n}}x_{2}\rightarrow y_{2}$. Hence, $y_{2} + \lambda\mu x_{2}=0$. Setting $y_{n1}=y_{1}$,
$x_{n1}=x_{1}$, $x_{n2}=x_{2}$ we solve again the corresponding system of the linear equations. Therefore
\[
\begin{array}{l}
x_{n3}=(y_{1}-\lambda^{k_{n}}x_{1}-k_{n}\lambda^{k_{n}-1}x_{2})\frac{2}{k_{n}(k_{n}-1)\lambda^{k_{n}-2}}\rightarrow 0\\
y_{n2}=\lambda^{k_{n}}x_{2} + \frac{2\lambda}{k_{n}-1}(y_{1}-\lambda^{k_{n}}x_{1}) - \frac{2\lambda k_{n}}{k_{n}-1}\lambda^{k_{n}-1}x_{2}\rightarrow -\lambda\mu x_{2}=y_{2}\\
y_{n3}=\lambda^{k_{n}}x_{n3}\rightarrow 0
\end{array}
\]
since $x_{n3}\rightarrow x_{3}=0$ and this finishes the proof of the theorem.
\end{proof}

\begin{proposition}
Let $A$ be a Jordan block with an eigenvalue $|\lambda | >1$.
\begin{itemize}
\item[(i)] If $x\neq 0$ then $J(x)=\emptyset$.

\item[(ii)] If $x=0$ then $J(0)=\mathbb{C}^{l}$.
\end{itemize}
\end{proposition}
\begin{proof}
We give the proof for the case $l=3$. Let $y=(y_{1},y_{2},y_{3})\in J(x)$ for some $x,y\in \mathbb{C}^3$. Then
there exist a sequence $\{ z_n\}$ in $\mathbb{C}^3$ and a strictly increasing sequence $\{ k_n\}$ of positive
integers such that $z_{n}=(x_{n1},x_{n2},x_{n3})\rightarrow (x_{1},x_{2},x_{3})$ and $A^{k_{n}}z_{n}\rightarrow
y$. Set $y_{n}=(y_{n1},y_{n2},y_{n3})=A^{k_{n}}z_{n}$. Since $y_{n3}= \lambda^{k_{n}}x_{n3}\rightarrow y_{3}$
then $k_{n}(k_{n}-1)x_{n3}\rightarrow 0$. From $y_{n2}\rightarrow y_{2}$ we get
\[
\frac{y_{n2}}{k_{n}}=\frac{\lambda^{k_{n}}}{k_{n}^{2}}k_{n}x_{n2} +  \lambda^{k_{n}-1}x_{n3} \rightarrow 0.
\]
Using the fact that $\lambda^{k_n}x_{n3}\rightarrow y_{3}$ it follows that the sequence $\{
\frac{\lambda^{k_{n}}}{k_{n}^{2}}k_{n}x_{n2} \}$ converges to a finite complex number, hence
$k_{n}x_{n2}\rightarrow 0$. The last implies $x_{n2}\rightarrow 0$, therefore $x_{2}=0$. We have
\[
x_{n1}=\frac{y_{n1}}{\lambda^{k_{n}}}-\frac{1}{\lambda}k_{n}x_{n2} -\frac{1}{2}\lambda^{2}k_{n}(k_{n}-1)x_{n3}.
\]
Observing that each one term on the right hand side in the previous equality goes to $0$ ($y_{n3}\rightarrow y_{3}$) we arrive at $x_{1}=0$. Therefore $x=0$. Hence by
Theorem \ref{jzerogreater}, $J(0)=\mathbb{C}^{3}$.
\end{proof}

The proof of the next proposition is trivial and therefore is omitted.

\begin{proposition}
If $A$ is a Jordan block with an eigenvalue $|\lambda | <1$ then $J(x)=\{ 0\}$ for every $x\in \mathbb{C}^{l}$.
\end{proposition}

As we have already mentioned in the introduction, hypercyclicity is a phenomenon which occurs in infinite dimensions. However, as Feldman showed in \cite{Fel}, there
exist $n\times n$ commuting complex matrices $A_1, A_2, \ldots , A_{n+1}$ and vectors $x\in \mathbb{C}^n$ such that the set
$$\{ A_1^{k_1}A_2^{k_2}\ldots A_{n+1}^{k_{n+1}}x: k_1,\ldots ,k_{n+1}\in \mathbb{N} \cup \{ 0\} \}$$ is dense in $\mathbb{C}^n$. For results towards this direction see
also \cite{CoHaMa1}. Now in a similar manner as above, one can define the extended limit set of a vector $x\in \mathbb{C}^n$ associated to a fixed finite set of
(commuting) matrices. For precise definitions we refer to \cite{CoHaMa2}. A quite demanding question is the following

\medskip
\noindent\textbf{Question.} Fix $k,n\geq 2$ positive integers. Let $A_1,A_2,\ldots ,A_k$ be $n\times n$ commuting complex matrices. Describe the $J$-set
$J_{(A_1,A_2,\ldots,A_k) }(x)$ for every $x\in \mathbb{C}^n$.

\end{document}